\documentclass[a4paper]{ejm}

\usepackage{amsmath,amsfonts,overpic,epsfig}
\usepackage{graphicx,epstopdf}

\newcommand{\E}{\mathbb{E}}

\newcommand{\Prob}{\mathbb{P}}
\newcommand{\Ind}{\mathbf{I}}
\newcommand{\Fb}{\bar{F}}
\newcommand{\fb}{\bar{f}}
\newcommand{\Gb}{\bar{G}}
\newcommand{\gb}{\bar{g}}

\newcommand{\I}{\mathcal{I}}
\newcommand{\D}{\mathcal{D}}

\newtheorem{thm}{Theorem}[section]

\begin{document}

\title{A stochastic analysis of greedy routing in a spatially dependent sensor network}

\author{Holger P. Keeler}

\maketitle

\section*{Abstract}
For a sensor network, a tractable spatially dependent node deployment model is presented with the property that the density is inversely proportional to the sink distance. A stochastic model is formulated to examine message advancements under greedy routing in such a sensor network. The aim of this work is to demonstrate that an inhomogeneous Poisson process can be used to model a sensor network with spatially dependent node density. Symmetric elliptic integrals and asymptotic approximations are used to describe the random behaviour of hops. Types of dependence that affect hop advancements are examined. We observe that the dependence between successive jumps in a multihop path is captured by including only the previous forwarding node location. We include a simple uncoordinated sleep scheme, and observe that the complexity of the model is reduced when sufficiently many nodes are asleep. All expressions involving multidimensional integrals are derived and evaluated with quasi-Monte Carlo integration methods based on Halton sequences and recently developed lattice rules. An importance sampling function is derived to speed up the quasi-Monte Carlo methods. The ensuing results agree extremely well with simulations.

\section{Introduction}
Advancements in networking technologies are leading to sensor networks being a feasible and common technology. Sensor networks consist of electronic sensing devices known as \emph{sensor nodes}. The nodes are deployed over a region known as a sensor field to gather environmental information. Each node has the ability to collect and process environmental data within its sensing range, and to communicate with other nodes within its transmission range. The collected data is ultimately relayed, often via surrounding nodes, to a main station known as a \emph{sink}. The applications of sensor networks are valuable and diverse, and they include security and traffic surveillance, environmental and animal monitoring, natural disaster warning and analysis, and building and structure assessment \cite{AKYILDIZ:2002,CHONG:2003,TUBAISHAT:2003}. 

A significant issue for sensor networks is developing routing methods that can handle their dynamic topologies. A common approach in sensor networks and ad hoc networks in general is geometric or position-based routing \cite{MAUVE:2001,STOJ:2002}. The operation of these algorithms is based on the assumption that each node knows its geographical location in relation to the sink, and the location of neighbouring nodes within its transmission radius. Geometric routing is often considered attractive because of its localized nature and scalability. We refer to the node from where a data message originates as the \emph{source} node. A natural geometric routing approach is for this source node to forward a message to the node that is within the source node's transmission range and geographically the closest to the sink, and to repeat this step until the message finally reaches the target sink. This approach serves as a \emph{greedy} routing method in itself or forms the basis for more intelligent routing methods in wireless ad hoc networks \cite{STOJ:1999,KUHN:2003,KARP:2000,ZORZI1:2003}.

Sensor nodes are often randomly scattered over the sensor field. To conserve power-consumption, a subset of sensor nodes may randomly fall into a low energy-consuming \emph{sleep} state, in which they cannot relay messages. The inherent randomness in sensor networks motivates the need for a suitable stochastic model to determine the ability of a routing scheme to successfully deliver data. In recent years, stochastic models and methods are being increasingly employed in analyzing communication networks. In particular, there is a special issue \cite{HABDF:2009} on stochastic geometry and related fields applied to communication networks, as well as a two-volume monograph by Baccelli and Blaszczyszyn \cite{BB1:2009,BB2:2009}.

The majority of this work, however, is under the assumption that sensor nodes are deployed according to a homogeneous Poisson process. Although such a tractable model might not capture the underlying node density variation, it can serve as a first approximation for studying network characteristics. We wish to extend the standard model by examining inhomogeneous node deployment such that the node density is spatially dependent.

There are various reasons why inhomogeneous deployment models may be necessary. The deployment of sensor nodes depends on the environment and application of the sensor network. Consequently, sensor nodes may need to be deployed more densely in important sensing regions. The obstacles in the network surroundings may prevent nodes from being positioned in certain regions. Inconsistent system parameters (such as battery lifetime) and interference can reduce the effective node density in certain regions. Furthermore, the nature of the actual node deployment influences the node density. For example, an aerial dispersal of nodes could result in the nodes obeying some type of diffusion process. Also, there may be network protocols that require positioning the nodes in certain regions, which lead to performance advantages.

The design and deployment of sensor networks must address the problem of message collisions. The data messages in sensor networks converge towards the sink. Nodes closer to the sink need act as relay nodes more than nodes away from the sink. One possible solution to this problem is to deploy more nodes in these heavy traffic regions. Consequently, the node density would decay at some rate that is dependent on the distance to the sink.

It is in this last setting that we wish to examine greedy routing. We propose an approach using a tractable mathematical model similar to the one developed in previous work \cite{KEELER1,KEELER2}. Ideally, we want to offer a computationally quick and reliable way to obtain probabilistic descriptions of multihop paths in a sensor network with a simple stochastic sleep scheme. Moreover, we want to extend the model, analysis, and calculations methods from the homogeneous case to the inhomogeneous case, and demonstrate that the techniques are still valid.

To achieve these goals, we analyze greedy routing in randomly deployed networks under the multihop situation. We propose a tractable spatially dependent node density function, and analyze the resulting stochastic characteristics. Furthermore, we examine the influence of a simple stochastic sleep scheme. More specifically, we examine the effects of having a certain proportion of nodes awake at any given time. We derive probability distributions that involve multiple integrals, and demonstrate their feasible evaluation via quasi-Monte Carlo integration methods based on Halton sequences and recently developed lattice rules.

The work presented here is focused on the stochastic behaviour of multihop paths, the calculation methods, and mathematically representing the effects of a sleep scheme. Overall, we show the application of this mathematical formulation in describing the stochastic behaviour of message delivery in sensor networks with inhomogeneous node deployment. Additionally, we demonstrate a calculation procedure based on quasi-Monte Carlo methods for evaluating multidimensional integrals.

\section{Background work}
The results presented here follow on from initial homogeneous Poisson model development and analysis \cite{KEELER1}, which was later extended by evaluating resulting integrals and analyzing a simple sleep scheme \cite{KEELER2}. Consequently, the majority of the formulation and solution techniques used here have been applied in the constant density setting.

Ishizuka and Aida \cite{IA:2004} performed analysis on node placement approaches via simulations to gauge the fault tolerance of networks against random node failure and battery exhaustion. In particular, they assumed that individual nodes were scattered uniformly around the sink, and that the density decreased as the distance to the sink increased. In addition to the homogeneous model, Ishizuka and Aida examined two models where in one the density is inversely proportional to the sink distance, and in the other the density is a Gaussian function (its standard deviation was chosen such that ninety-nine percent of nodes were found in the test region). They concluded that the simple inverse function resulted in the best fault tolerance overall. Ishizuka and Aida \cite{IA:2007} later examined a more general power-law model, and concluded again that the sensor networks are the most resilient to node failure when the density inversely proportional to sink distance.

A concept closely related to connectivity is the sensing coverage of a sensor network, which is the ability of a sensor network to successfully sense or cover the entire sensor field. Solutions to coverage problems have been based on coverage processes such as the Boolean model; see Hall \cite{HALL:1988} and Stoyan, Kendall and Mecke \cite{SKM:1995} for more details. In the sensor network setting, a more recent example is that of Pallavi et al. \cite{MRM:2009} who used coverage processes to examine the coverage of a sensor network with an exponentially decreasing node density.

The aforementioned work involving inhomogeneous node deployment cases did not examine the stochastic behaviour of any particular routing method. Furthermore, the work did not cover the effects that a sleep scheme has on stochastic dependencies in the routing model.

\section{Mathematical model}
We present a two-dimensional model that neglects the earth's curvature. We assume that nodes communicate data radially, and that a node's transmission radius is a constant that clearly cuts off at some distance, which implies that a node can relay data only to another node when it is within the forwarding node's transmission radius. For numerical calculations and simulations, the transmission radius is set to one. However, we denote the transmission radius $r$ in ensuing calculations and equations for clarity and future extensions.

We assume that nodes are scattered according to a two-dimensional Poisson process over a finite region, and that at any given time a random number of nodes are in sleep mode while the remaining are in their awake mode. Furthermore, we assume the awake node density, or the average number of awake nodes per unit area, is a spatially dependent function; that is, the nodes are scattered according to an inhomogeneous Poisson process \cite{SKM:1995}. We assume the nodes are scattered uniformly around the sink, but decreases in some way as the distance to the sink increases. Hence, the awake node density is a radial function of the form
\begin{equation}\label{inhomlambda0}
\lambda(u)=\lambda q(u)\qquad u\in(0,\infty),
\end{equation}
where $u$ is the distance to the sink, $q(u)$ is a non-negative shaping function, and $\lambda$ is a postive constant. The function $\lambda(u)$ can be interpreted as the mean number of awake nodes per infinitesimal area element. We refer to the constant $\lambda$ as the \emph{initial} node density, and use it to scale the density function.

Ideally, the $q(u)$ function should be amenable to analytical and asymptotic methods while reflecting a realistic node placement. We study a node density that is inversely proportional to the sink distance
\begin{equation}\label{inhomlambda}
q(u)=\frac{1}{u},
\end{equation}
which is the model proposed by Ishizuka and Aida \cite{IA:2004}. Their simulation work showed that overall this node deployment model outperformed a Gaussian model in both tolerance against battery exhaustion and random node failure. We also believe that this positioning of nodes will better accommodate the convergence of messages near the sink. Consequently, we examine this model due its observed tolerance in simulations, while still being tractable to mathematical methods.

We introduce a Poisson process \emph{mean} measure $\Lambda$, which for a bounded Borel set $B\subset\mathbb{R}^2$ with area $A$ is the density function integrated over the region $B$, thus in polar coordinates
\begin{align}
\Lambda(B)&=\lambda \iint_Bq(u)udud\theta,\\
&= \lambda Q(B),
\end{align}
where $Q$ is referred to as the \emph{rescaled} mean measure, and is used such that the notation is analogous to the results based on the constant density case when $Q$ reduces to the area $A$ of the region \cite{KEELER1,KEELER2}. We emphasize that all the $Q$-type expressions we consider in this work are derived in a similar manner to the corresponding area expressions under the homogeneous model by integrating the density over specific regions \cite{KEELER1,KEELER2}.

Under a sleep model, the initial node density parameter can be written as $\lambda=p\alpha$ where $p$ is the probability of a node being awake and $\alpha$ is the underlying (that is, sleep and awake) node density parameter. The awake nodes form a thinned Poisson process. As in the homogeneous case, both the number of points kept and the number of points removed form random variables that are independent of each other \cite{SKM:1995}. Hence, in our model the number of awake nodes is independent of the number of asleep nodes, which examined further in Section \ref{sleep5}.

The number of awake nodes $N_B$ located within some region $B$ is a inhomogeneous Poisson random variable with a probability mass function
\begin{equation}\label{inhomPoisson1}
\mathbb{P}(N_B=n)=\frac{\left(\lambda Q(B)\right)^n}{n!}e^{-\lambda Q(B)}.
\end{equation}

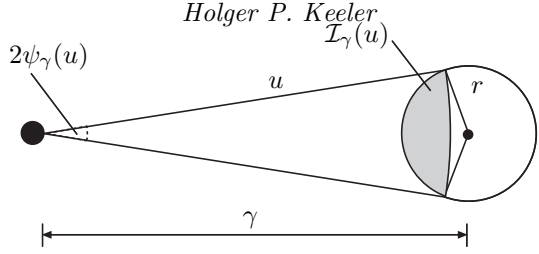
\begin{figure}
\begin{center}
\begin{overpic}[scale=0.60]{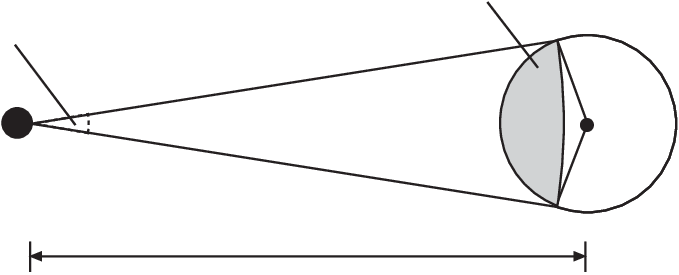}
\put(59,40){$\I_{\gamma}(u)$}
\put(48,31){$u$}
\put(87,30){$r$}
\put(43,5){$\gamma$}
\put(-02,36){$2\psi_{\gamma}(u)$}
\end{overpic}
\caption{Integrate over the region $\I_{\gamma}(u)$ to derive $\Lambda(\I_{\gamma}(u))$ for the forwarding node.\label{overlapcircle5}}
\end{center}
\end{figure}

\section{Single hop analysis}
We introduce the parameter $\gamma$ to represent the distance between a forwarding node and the sink (respectively located on the right and the left side in Fig. \ref{overlapcircle5}). Often we shall present results such that the initial node density parameter $\lambda$ is a product of the source-sink distance and some positive constant. For a forwarding node with a sink distance $\gamma$, let $\I_{\gamma}(u)\subset\mathbb{R}^2$ be its partial feasible region as a function of $u$. This region is formed by the intersection of two circles of radii $r$ and $u$ centered at the source node and the sink respectively (as illustrated on the right in Fig. \ref{overlapcircle5}). The area of $\I_{\gamma}(u)$ is denoted by $A_{\gamma}(u)$, and given by the integral
\begin{equation}\label{area5}
A_{\gamma}(u)=2\int_{\gamma-r}^u \int_0^{\psi_\gamma(w)}wd\theta dw,
\end{equation}
where the sink angle function $\psi_{\gamma}(u)$, by the law of cosines, is defined as
\begin{equation}\label{psi5}
  \psi_{\gamma}(u)=\arccos\left(\frac{u^2+\gamma^2-r^2}{2u\gamma}\right).
\end{equation}
Furthermore, the mean measure of the region $\I_{\gamma}(u)$ is given by the integral
\begin{equation}\label{Lambda0}
\Lambda(\I_{\gamma}(u))=2\int_{\gamma-r}^u \int_0^{\psi_\gamma(w)}\lambda(w)wd\theta dw,
\end{equation}
which for the functions (\ref{inhomlambda0}) and (\ref{inhomlambda}) reduces to
\begin{align}\label{Lambda1}
\Lambda(\I_{\gamma}(u))&=2\lambda\int_{\gamma-r}^u \int_0^{\psi_\gamma(w)}d\theta dw\\
&=2\lambda\int_{\gamma-r}^u \psi_\gamma(w)dw.\label{Lambda2}
\end{align}
Henceforth, we write
\[
\Lambda_{\gamma}(u)= \Lambda(\I_{\gamma}(u)), \qquad
Q_{\gamma}(u)= Q(\I_{\gamma}(u)),
\]
to refer to the mean measure and the rescaled mean measure respectively, the motivation of which will become apparent by the analogous results that follow.

To calculate the mean measure $\Lambda_{\gamma}(u)$, we write integral (\ref{Lambda2}) as
\begin{equation}
2 \lambda\int_{\gamma-r}^u\psi_{\gamma}(w)dw=2\lambda u\psi_{\gamma}(u)+2\lambda\int_{\gamma-r}^u\frac{(w^2+r^2-\gamma^2)}{\left([w^2-(\gamma-r)^2][(\gamma+r)^2-w^2]\right)^{1/2}}dw,
\end{equation}
whose solution is obtained with the reduction of general elliptic integrals by symmetric elliptic integrals; see Section 19.29 of \cite{DLMF:2010} for further details and examples. The final solution of (\ref{Lambda2}) is 
\begin{align}
\nonumber\Lambda_{\gamma}(u)= &\phantom{o} 2\lambda [u \psi_{\gamma}(u)- \frac{1}{3}(\gamma-r)^2(\gamma+r)^2R_D(v^2+(\gamma-r)^2,v^2+(\gamma+r)^2, v^2) 
    +u(\gamma-r)/v\\
    &+(r^2-\gamma^2)R_F(v^2+(\gamma-r)^2,v^2+(\gamma+r)^2, v^2)],\qquad u\in(\gamma-r,\gamma]\label{meanU},
\end{align}
where 
\begin{equation}
v=\frac{\gamma-r}{u^2-(\gamma-r)^2}\left([u^2-(\gamma-r)^2][(\gamma+r)^2-u^2)]\right)^{1/2},
\end{equation}
and $R_F$ and $R_D$ are symmetric elliptic integrals in Carlson form
\begin{align}
&R_F(x,y,z)=\frac{1}{2}\int_0^{\infty}\frac{dt}{\left[(t+x)(t+y)(t+z)\right]^{1/2}},
\\
&R_D(x,y,z)=\frac{3}{2}\int_0^{\infty}\frac{dt}{(t+z)\left[(t+x)(t+y)(t+z)\right]^{1/2}}.
\end{align}
For calculation purposes, we note that naturally
\begin{equation}
\lim_{u\rightarrow\gamma-r}\Lambda_{\gamma}(u)=0.
\end{equation}
The results presented here are obtained via our purposely-written elliptic integrals based on the papers by Carlson \cite{CARLSON:1979,CARLSON:1995}, which give algorithms that are readily implementable and can handle both real and complex values under specified variable and parameter regimes. We found that the solutions are calculated quickly owing to the speedy convergence of the Carlson algorithms.

Take as our sample space $\Omega$ the set of two-dimensional point processes on $\mathbb{R}^2$, together with an appropriate $\sigma$-field ${\cal F}$ of subsets of $\Omega$ and the probability measure $\Prob$ induced on this space by our Poisson intensity measure (\ref{meanU}). For a node at distance $\gamma$ from the origin, let the random variable $U$ with resepect to $(\Omega,{\cal F}, \Prob)$ be the sink distance of the forwarding node after a single message hop. To derive the probability distribution of $U$, the nearest neighbour \cite{SKM:1995} approach is used in which $\Prob(U > u$) is equated to the probability of no nodes existing in the feasible region at a distance less or equal to $u$. This argument is analogous to that of the homogeneous case \cite{KEELER1}, hence the distribution
\begin{equation}\label{inhomFu}
F_{\gamma}(u) = \left\{ \begin{array}{ll}
1-e^{-\lambda Q_{\gamma}(u)} & \gamma-r \leq u <\gamma\\
1 & u\geq \gamma\\
0& u<\gamma-r,
\end{array} \right.
\end{equation}
follows. There is a jump discontinuity in the distribution at $u=\gamma$ owing to the positive probability that no nodes lie within the feasible region. On the support where the distribution (\ref{inhomFu}) is absolutely continuous the probability density is defined by
\begin{equation}\label{inhomfu}
f(u)=\lambda Q_{\gamma}'(u)e^{-\lambda Q_{\gamma}(u)},\qquad \gamma-r \leq u < \gamma,
\end{equation}
where the derivative of the rescaled mean measure
\[
Q_{\gamma}'(u)=2\psi_{\gamma}(u).
\]

Let $C=\gamma-U$ be the distance advanced towards the sink when the originating node is at a distance $\gamma$. Let $\Fb$ denote the distribution of $C$, which leads to
\begin{equation}\label{inhomFc}
\Fb_{\gamma}(c) = \left\{ \begin{array}{ll}
e^{-\lambda Q_{\gamma}(\gamma-c)} & 0 < c \leq r\\
1 & c> r\\
0& c\leq 0,
\end{array} \right.
\end{equation}
and its probability density is given by
\begin{equation}\label{inhomfc}
\fb(c)=\lambda Q_{\gamma}'(\gamma-c)e^{-\lambda Q_{\gamma}(\gamma-c)},\qquad 0 < c \leq r.
\end{equation}
 Given that $C$ is non-negative, the $m$-th moment expression
\begin{align}\label{inhomeCm}
\mathbb{E}(C^m) & = m\int_0^r c^{m-1}\Prob(C>c)dc,\\
& =  r^m-m\int_{0}^r c^{m-1} e^{-\lambda Q_{\gamma}(\gamma-c)}dc,
\end{align}
follows.
 
\subsection{Asymptotic results}
We derive an asymptotic approximation to the rescaled mean measure $Q_{\gamma}(u)$, which gives a tractable and accurate expression. The feasible region of the forwarding node approaches zero at the point $u=\gamma-r$, and is the point closest to the sink. The angle function $\psi_{\gamma}$ expanded at this point gives
\begin{equation}\label{approxpsi5}
\psi_{\gamma}(u)\approx b_0(u-\gamma+r)^{1/2}+b_1(u-\gamma+r)^{3/2}+
b_2(u-\gamma+r)^{5/2},
\end{equation}
where the expansion terms
\begin{align}
b_0 &= \left[\frac{2r}{\gamma (\gamma-r)}\right]^{1/2},\\
b_1 &= \left[\frac{2r}{\gamma (\gamma-r)}\right]^{1/2}\left[\frac{r^2-3 r\gamma -3 \gamma^2}{12 (\gamma^2 r- \gamma r^2)}\right],\\
b_2 &= \left[\frac{2r}{\gamma (\gamma-r)}\right]^{1/2}\left[\frac{3 r^4+25 r^2\gamma^2 -10r^3\gamma +30 \gamma^3 r-5 \gamma^4}{160 \gamma^2 (\gamma-r)^2 r^2}\right],
\end{align}
follow \cite{KEELER1}. The asymptotic result for the rescaled mean measure
\begin{equation}\label{approxQ}
Q_{\gamma}(u)\approx 4\left[\frac{b_0}{3}(u-\gamma+r)^{3/2}+\frac{b_1}{5}(u-\gamma+r)^{5/2}+
\frac{b_2}{7}(u-\gamma+r)^{7/2}\right],
\end{equation}
follows.

\begin{figure}
\begin{center}
\includegraphics[scale=0.6]{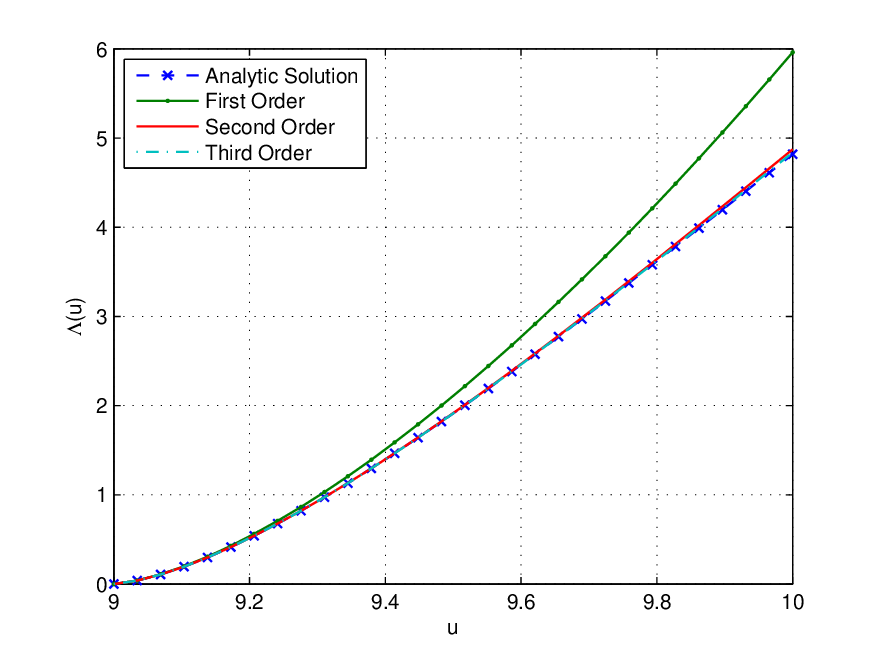}
\caption{Comparison of $\Lambda_{\gamma}(u)$ asymptotic expressions ($\lambda=3\gamma$, $r=1$ and $\gamma=10$) .\label{CompLambdaR1}}
\end{center}
\end{figure}

The results of the second-order approximation are accurate for a unit radius (see Fig. \ref{CompLambdaR1}). Adding the third term only improves the results slightly. However, it may be needed for larger transmission radius models (as Fig. \ref{CompLambdaR5} suggests). Conversely, the two-term expansion gives accurate results (see Fig. \ref{AsymHopDist}) when substituted into the sink distribution (\ref{inhomFu}). In fact, it appears that the expansion (\ref{approxQ}) consisting of elementary functions can be used to give accurate results, which are clearly faster to evaluate than those based on elliptic integrals. However, we continue to use the exact solution of the integral (\ref{Lambda2}), and later compare it to its approximation.

\begin{figure}
\begin{center}
\includegraphics[scale=0.6]{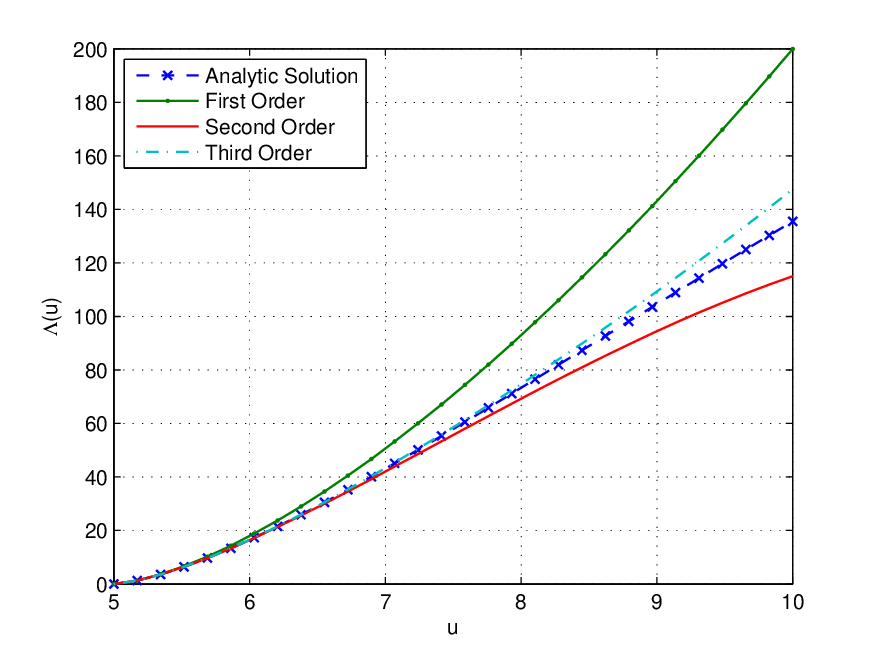}
\caption{Comparison of $\Lambda_{\gamma}(u)$ asymptotic expressions ($\lambda=3\gamma$, $r=5$ and $\gamma=10$) .\label{CompLambdaR5}}
\end{center}
\end{figure}

\begin{figure}
\begin{center}
\includegraphics[scale=0.6]{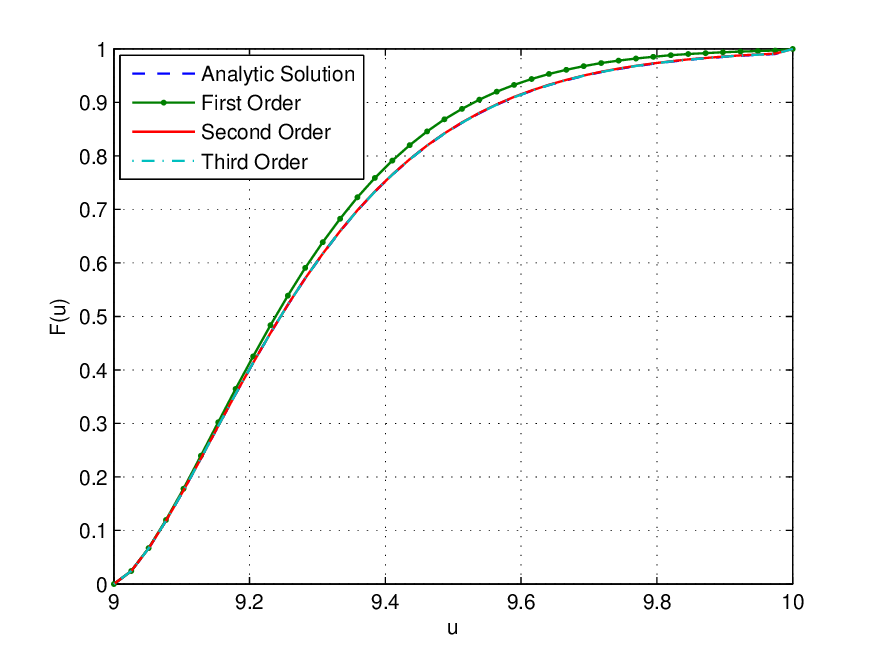}
\caption{Comparison of $F_{\gamma}(u)$ asymptotic expressions ($\lambda=3\gamma$, $r=1$ and $\gamma=10$) .\label{AsymHopDist}}
\end{center}
\end{figure}

\begin{figure}
\begin{center}
\includegraphics[scale=0.6]{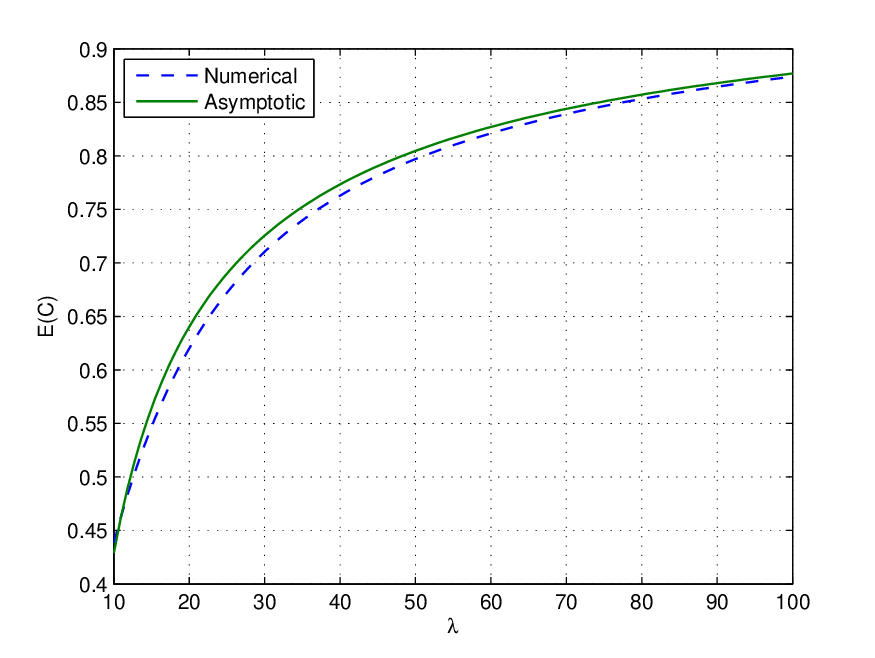}
\caption{Numerical and asymptotic results of first moment $\E(C)$ ($r=1$ and  $\gamma=10$).\label{AsymMomemt1}}
\end{center}
\end{figure}

For large $\lambda$, we present asymptotic moment results for our spatially dependent node density model.
\begin{thm}\label{AsymThm5}
For the mean measure (\ref{meanU}), provided $\gamma>r$, under greedy routing the first hop moment
\begin{equation}\label{aEc5}
\mathbb{E}(C)\sim r-\dfrac{\Gamma(5/3)}{\left(\lambda q_0\right)^{2/3}} ,
\end{equation}
and the second hop moment
\begin{equation}\label{aVc5}
\mathbb{E}(C^2)\sim r^2-2r\dfrac{\Gamma(5/3)}{\left(\lambda q_0\right)^{2/3}}+\dfrac{\Gamma(7/3)}{\left(\lambda q_0\right)^{4/3}} ,
\end{equation}
as the initial node density
\[
\lambda \rightarrow \infty,
\]
where $\Gamma$ is the gamma function, and
\[
q_0=\dfrac{4}{3}\left[\frac{2r}{\gamma (\gamma-r)}\right]^{1/2}.
\]
\end{thm}

\begin{proof}
Consider integrals of the form
\begin{equation}
I(\lambda)=\int_a^b \phi(t) e^{-\lambda Q(t)}dt,
\end{equation}
Assume the real function $Q(t)$ has one minimum on the interval $[a,b]$, which occurs at $t=a$, and that
\begin{equation}\label{exp1}
Q(t)\sim Q(a)+\sum_{s=0}^{\infty}q_s(t-a)^{s+\mu},
\end{equation}
and
\begin{equation}\label{exp2}
\phi(t)\sim \sum_{s=0}^{\infty}\phi_s(t-a)^{s+\beta-1},
\end{equation}
as $t\rightarrow a^+$, and $\lambda$ and $\mu$ are positive constants, and the constant $\beta$ can be real or complex provided that the real part is positive; for more details see Laplace's method \cite[page 58]{WONG:1989}. Furthermore, assume that the first expansion (\ref{exp1}) can be differentiated
\begin{equation}\label{exp3}
Q'(t)\sim \sum_{s=0}^{\infty}(s+\mu)q_s(t-a)^{s+\mu-1},
\end{equation}
as $t\rightarrow a^+$. Provided that $Q(t)>Q(a)$ for all $t\in$ $(a,b)$, then Laplace's method can be applied to integrals of the form
\begin{equation}
I(\lambda)=\int_a^b (t-a)^k e^{-\lambda Q(t)}dt,
\end{equation}
where under our setting $k=0$ or $k=1$, and $a=0$, thus giving
\[
I(\lambda)\sim \dfrac{\Gamma(\tau)}{\mu(\lambda q_0)^{\tau}},\quad \lambda \rightarrow \infty,
\]
where
\[
\tau=\dfrac{2(k+1)}{3}.
\]
Since $Q(a)$ needs to be the minimum on the integral interval, we use the change of variable $t=u-\gamma+r=r-c$ in the expansion of $Q(t)$, which, with a slight abuse of notation, leads to
\begin{equation}\label{asymareaT}
\begin{array}{c}
  Q_{\gamma}(t)\sim q_0t^{3/2}+O(t^{5/2})
\quad\textrm{as }\quad t \rightarrow 0,
\end{array}
\end{equation}
and 
\[
q_0=\dfrac{4b_0}{3},\qquad \mu=\dfrac{3}{2}.
\]
The first moment result (\ref{aEc5}) follows by setting $k=0$. The change of variable applied to the second moment equation gives
\[
\mathbb{E}(C^2)=r^2-2\int_{0}^r(r-t)e^{-\lambda Q_{\gamma}(t)}dt,
\]
which leads to the second result (\ref{aVc5}) by setting $k=0$ and $k=1$ accordingly.
\end{proof}

\subsection{Sink dependence}\label{sinkDep}
Since $\gamma$ is the distance of an arbitrary node forwarding a message, we set $\gamma=\ell$ when the forwarding node is the source node. The node intensity function is clearly dependent on the source node sink distance. Comparing the hop distributions of messages from two different sources (in Fig. \ref{CompHopDist5}) reveals that a message is relayed farther in a single hop if the forwarding node is closer to the sink. Intuitively, hops increase stochastically as the message approaches the sink as more potential forwarding nodes are available in the forwarding region. Geometrically, we observe that the integral kernel in the mean measure equation (\ref{Lambda2}) is the angle function $\psi_{\gamma}$, which decreases as $\gamma$ increases; that is
\begin{equation}\label{psiIneq}
\psi_{\gamma_1}(u) \leq \psi_{\gamma_2}(u), \qquad \gamma_1\geq \gamma_2 \geq r.
\end{equation}
Conversely, $\psi_{\gamma}$ and, hence, the integral (\ref{meanU}) increases as the sink distance decreases. This dependence on the sink distance of the forwarding node is simply referred to as the \emph{sink} dependence.

The influence of the sink dependence can be observed by comparing the difference in two hop distributions with different sink distances $\gamma_1$ and $\gamma_2$. In previous work \cite{KEELER1}, the hop distribution dependence on the sink distance under the homogeneous model was examined by a Kullback-Leibler. The Kullback-Leibler divergence \cite{KULLBACK:1959}, also known as relative entropy, is an asymmetric measure of the difference between two probability distributions, and it applied to the mixed discrete-continuous hop distribution gives
\begin{equation}\label{KLInhom}
D(\gamma_1,\gamma_2)=\int_0^r \fb_{\gamma_2}(c)\log\left[\frac{\fb_{\gamma_2}(c)}{\fb_{\gamma_1}(c)} \right]dc+\Fb_{\gamma_2}(0^+)\log\left[\frac{\Fb_{\gamma_2}(0^+)}{\Fb_{\gamma_1}(0^+)} \right],
\end{equation}
where the routing void (that is, no nodes in the feasible region) probability
\begin{equation}
\Fb_{\gamma}(0^+)= e^{-\lambda Q_{\gamma}(\gamma)}.
\end{equation}
The Kullback-Leibler divergence is non-negative and is zero for identical distributions \cite{KULLBACK:1959}.
\begin{figure}
\begin{center}
\includegraphics[scale=0.6]{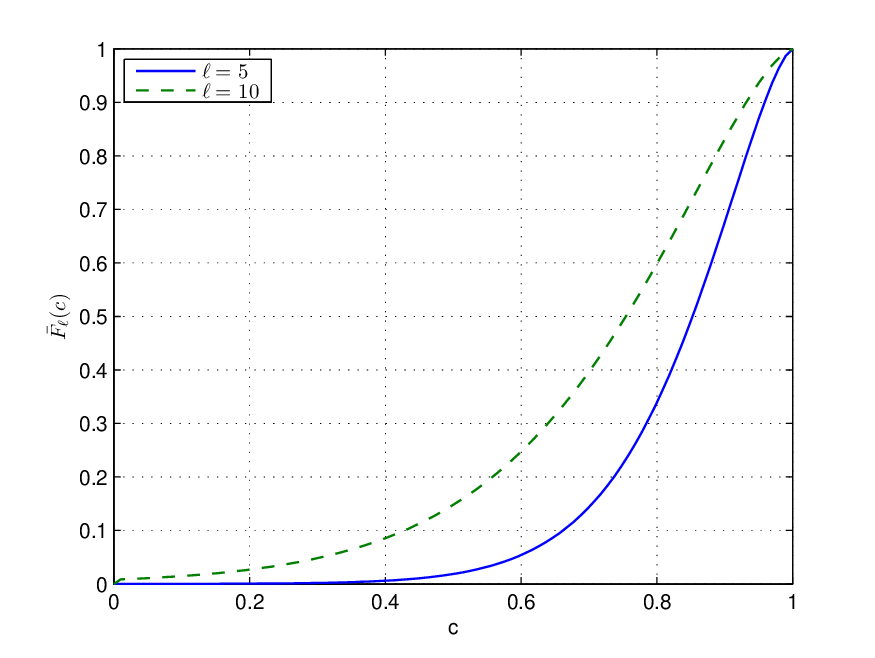}
\caption{Comparison of hop distributions $\Fb_{\ell}(C)$ for different $\ell$ ($\lambda=30$, $r=1$, and $\ell=5$ and $\ell=10$)).\label{CompHopDist5}}
\end{center}
\end{figure}

\begin{figure}
\begin{center}
\includegraphics[scale=0.6]{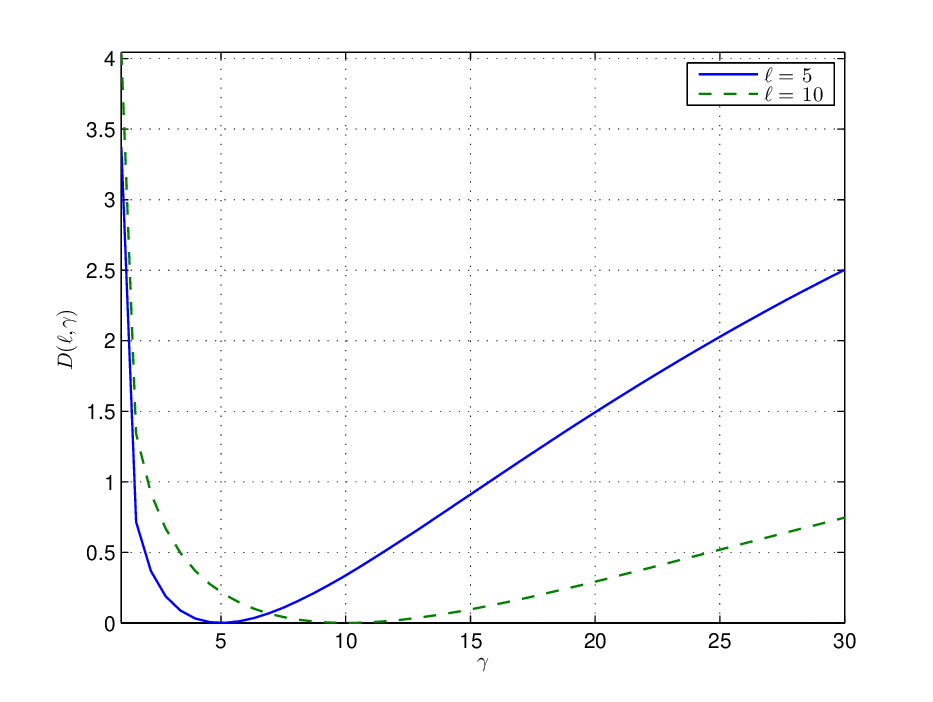}
\caption{Kullback-Leibler analysis of hop distributions for different $\ell$ ($\lambda=30$, $r=1$, and $\ell=5$ and $\ell=10$).\label{kullbackInhom}}
\end{center}
\end{figure}

We calculated the integral in equation (\ref{KLInhom}) numerically to observe how the hop distribution is influenced when we set $\gamma_1=\ell$ and vary $\gamma_2=\gamma$. The comparison reveals that $D(\ell,\gamma)$ is high near the sink and decreases as $\gamma$ increases (refer to Fig \ref{kullbackInhom}). This is the same expected behaviour as observed under the homogeneous model \cite{KEELER1}, however, $D(\ell,\gamma)$ decays relatively slowly under the inhomogeneous model. We observe after $D(\ell,\gamma)=0$ at $\gamma=\ell$, it then increases at a rate that depends on the value of $\ell$. This contrasts starkly with the homogeneous model where $D(\ell,\gamma)$ is mostly zero away from the sink, and only increases significantly near the sink \cite{KEELER1}. Furthermore, $D(\ell,\gamma)$ differs significantly for two different $\ell$; that is, the hop distribution varies with respect to the sink distance. This behaviour differs from the constant node density case where renewal processes can be used to model and bound message advancement over multihop routes due to the hop distribution varying only slightly \cite{ZORZI1:2003,KEELER1}; hence renewal processes cannot be used here. 

\section{Multihop analysis}
For a multihop analysis, we introduce indexing for the random variables $U$ and $C$ by initially setting $U_0=\ell$ and $U_1=U$. For a node at distance $\ell$ from the origin, the random variables $U_i$ that give the sink distance of the forwarding node after $i$ hops are defined with respect to $(\Omega,{\cal F}, \Prob)$. We define the $i$-th hop advancement as $C_i=U_{i-1}-U_i$. Each $C_i$ depends on the forwarding node's sink distance $U_{i-1}$; thus, the sink distance of the source node clearly affects the first hop and subsequent hops. As noted in the previous section, comparing the hop distributions demonstrates that each $C_{i}$ is stochastically dominated by $C_{i+1}$. That is, for $i\geq0$, we have the stochastic ordering
\begin{equation}\label{algineq}
\mathbb{P}(C_{i+1}>c)\geq \mathbb{P}(C_{i}>c) , \qquad c\in(0,r).
\end{equation}
This inequality is the opposite to the equivalent result under the homogeneous model as noted by Zorzi and Rao \cite{ZORZI1:2003} and Keeler and Taylor \cite{KEELER1}. Consequently, this stochastic ordering of $C_i$ is dependent on the choice of $q(u)$, and does not hold in general. However, if given a decreasing shaping function such that $q(u)\leq1/u$ for all $u$, then the inequality 
\begin{equation}\label{psiIneq2}
uq(u)\psi_{\gamma_1}(u) \leq uq(u)\psi_{\gamma_2}(u), \qquad \gamma_1\geq \gamma_2 \geq r,
\end{equation}
holds, and the kernel in the mean measure integral (\ref{Lambda0}) decreases as $\gamma$ increases, and thus, under these conditions the stochastic ordering (\ref{algineq}) holds.

\subsection{Path dependence}\label{pathDep}
Let the random variable $\Theta_i$ be the angle between the $i$-th node and the previous node in relation to the sink. We assign the point $X_i=(U_i,\Theta_i)$ to the $i$-th forwarding node. The source (or zeroth) node corresponds to the point $X_{0}=(\ell,0)$. A message travels $i$ hops along a path that corresponds to a sequence of random points $\vec{X}_i=(X_0,X_1,\dots,X_i)$.

Let $\I_i(u_{i+1})\subset\mathbb{R}^2$ be the feasible region of the $i$-th forwarding node as a function of $u_{i+1}$ under the independent model. After the first hop, the nature of greedy routing implies that another dependence arises in the distribution of $U_2$, which was observed under the homogeneous model \cite{KEELER1,KEELER2}. If a forwarding node is chosen, then there are no other nodes in the source feasible region closer to the sink. Hence, the intersection of the feasible regions of the source and the first node (that is, $\I_1\cap\I_{0}$ in Fig. \ref{stochdep}) has no awake nodes. This implies that $U_2$ is  dependent on both $U_1$ and $\Theta_1$, the angle between the first node and the source node in relation to the sink. We call this dependence in both the sink distance and the sink angle the \emph{path} dependence, and the hop model that includes both the path and the sink dependence simply the \emph{dependent} model. Conversely, the \emph{independent} model only includes the sink dependence.

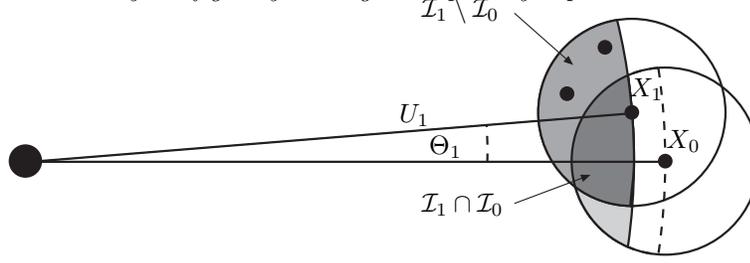
\begin{figure}
\begin{center}
\begin{overpic}[scale=.83]{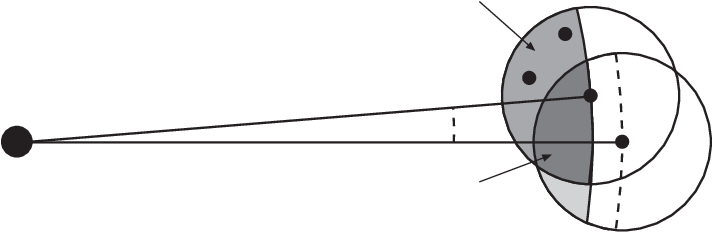}
\put(56,13.8){$\Theta_1$}
\put(52,18){$U_1$}
\put(87.5,14.5){$X_{0}$}
\put(82.6,21.4){$X_{1}$}
\put(55,32){$\I_1\setminus\I_{0}$}
\put(55,6.1){$\I_1\cap\I_{0}$}
\end{overpic}
\caption{No awake nodes in the intersection region $\I_1\setminus\I_{0}$ during the first message relay.\label{stochdep}}
\end{center}
\end{figure}

In the dependent model with no sleep scheme, the randomness is all encapsulated in the inhomogeneous two-dimensional Poisson process that gives the node locations, and everything else is deterministically given. Conversely, the independent model has a different source of randomness. The two-dimensional Poisson process is re-sampled at each time step.

Thus we have to extend underlying sample space to the set $\Omega^\infty$ of sequences of realizations of two-dimensional point processes, together with the probability measure induced by the assumption that the realizations are independent and compatible with the Poisson intensity measure (\ref{Lambda1}). The random variables $U_i$ are now defined to give the sink distance of the forwarding node after $i$ hops where, at each hop, we select the next node in the same manner as that described above for the dependent model, but according to the $i$-th realization of the spatial Poisson process in the sequence.

The assumption that there are independent realizations of the underlying node distribution at each hop may be thought to be unusual, but it has been implicitly assumed before \cite{ZORZI1:2003}, and it does lead to tractable approximations and bounds. Furthermore, it becomes a good model if there is a sleep scheme in operation in which nodes are alternately available and unavailable to act as transits.

Under a sleep scheme, if nodes were asleep during the previous message relay, it is possible for recently awoken nodes to be present in this region during the current message relay. To perform the initial analysis we assume no sleep scheme exists (by setting $p=1$). After our analysis, we include a simple sleep scheme and examine how varying $p$ and $\alpha$ (while fixing $\lambda$) affects the path dependence.

We denote the probability measures derived under the dependent and independent models respectively by $\Prob_D$ and $\Prob_I$ (the subscript is dropped if a result applies to both models). The distribution of $U_{i+1}$ under the independent model is dependent only on the sink distance of the current forwarding node, hence we write
\begin{equation}\label{Fi}
F_{i}(u_{i+1})=\mathbb{P}_I(U_{i+1}\leq
u_{i+1}|U_i=u_i),
\end{equation}
while under the dependent model the distribution is dependent on the message path, and so we write
\begin{equation}\label{Gi}
G_{i}(u_{i+1})=\mathbb{P}_D(U_{i+1}\leq
u_{i+1}|\vec{X}_i=\vec{x}_i).
\end{equation}
We denote the rescaled mean measure of the feasible region under the independent and dependent models respectively as $Q_{i}(u_{i+1})$ and $\vec{Q}_{i}(u_{i+1})$.

The rescaled mean measure under the independent model is always given by the original equation (\ref{meanU}), and hence, the distribution and  probability density of $U_{i+1}$ are obtained by setting $\gamma=u_i$ in equations (\ref{inhomFu}) and (\ref{inhomfu}). Under the dependent model, if the rescaled mean measure is given after $i$ hops, the sink distribution
\begin{equation}\label{Gu}
G_{i}(u_{i+1}) = \left\{ \begin{array}{ll}
1-e^{-\lambda \vec{Q}_i(u_{i+1})} &\quad u_{i}-r \leq u <u_{i}\\
1 &\quad u\geq u_{i}\\
0&\quad u<u_{i}-r,
\end{array} \right.
\end{equation}
immediately follows, and where it is absolutely continuous its  probability density
\begin{equation}\label{gu}
g_{i}(u_{i+1}) =
  \lambda \vec{Q}'_i(u_{i+1})e^{-\lambda \vec{Q}_i(u_{i+1})},
\end{equation}
also follows.

The set representing the feasible region under the dependent model follows by excluding the intersections of previous feasible regions, namely
\begin{equation}
\D_i(u_{i+1})=
\I_i(u_{i+1})\setminus\displaystyle\cup_{j=0}^{j=i-1}\I_j(u_{j+1}).
\end{equation}
Under the dependent model, it is possible to calculate the rescaled mean over the feasible region after one hop; see appendix for details. If $i\geq2$, we approximate the feasible region under the dependent model
\begin{equation}\label{DepSetApprox}
\D_i(u_{i+1})\approx
\I_i(u_{i+1})\setminus\displaystyle\I_{i-1}(u_{j+1}),
\end{equation}
where we refer to this approximation as the one-hop model. Results under the homogeneous model revealed that this approximation sufficiently captures the path dependence \cite{KEELER2}. Consequently, only the location of the previous node is needed to calculate the rescaled mean under the dependent model, and approximation (\ref{DepSetApprox}) is used henceforth.

The path dependence depends on the precise locations of the previous nodes while the sink dependence which only depends on their distance from the sink. To capture this observation and use it to describe the random behaviour of message hops we need the joint density of $U_i$ and $\Theta_{i}$. Under our inhomogeneous Poisson model, the joint probability density is
\begin{equation}\label{gjoint}
g_{i}(u_{i+1},\theta_{i+1})= \lambda_{\D_i} (u_{i+1},\theta_{i+1}) e^{-\lambda \vec{Q}_i(u_{i+1})},
\end{equation}
where the derivation is analogous to that of the homogeneous model \cite{KEELER2}; see appendix for details. The spatially dependent initial density $\lambda_{\D_i}(u_{i+1},\theta_{i+1})=\lambda\mathbf{I}_{\D_i}(u_{i+1},\theta_{i+1})$ and the indicator function of the dependent feasible region
\begin{equation}\label{indD}
\mathbf{I}_{\D_i}(u_{i+1},\theta_{i+1})= \left\{ \begin{array}{ll}
1, &\qquad  (u_{i+1},\theta_{0i})\in \D_i,\\
0,&\qquad \textrm{otherwise},
\end{array} \right.
\end{equation}
where angular coordinate $\theta_{0i}$ is the angle between the source node and the $i$-th forwarding node in relation to the sink. These expressions serve as the basis of the node density and the joint probability density under the sleep model in the next section.

\section{Sleep model}\label{sleep5}
We outline a simple sleep scheme that has been analyzed under the homogeneous model \cite{KEELER1}. We assume that the probability that a node is awake on each hop is $p$ and the event that a node is awake during a transmission attempt is independent of the event that it is awake at other transmission attempts. Consequently, the intersection region ($\I_1\cap\I_{0}$ in Fig. \ref{stochdep}) has a thinned initial node density  $(1-p)\lambda$. The rest of the feasible region $\I_1\setminus\I_{0}$ has an initial density $\lambda$. It follows that the initial node density function after one message hop is given by
\begin{equation}\label{sleeplambda}
\lambda_{\D_1}(u_{2},\theta_2)= \lambda\left[ \Ind_{\I_1\setminus\I_{0}}(u_{2},\theta_2)
+(1-p)\Ind_{\I_1\cap\I_{0}}(u_{2},\theta_2)
\right],
\end{equation}
where the superscripts denote the indicator functions of the disjoint regions.

In the limit as $p$ approaches zero and $\alpha$ approaches infinity with $\lambda=p\alpha$ held constant, the locations of the nodes after each hop is re-sampled, thus, completely removing the path dependence and allowing each forwarding node to be treated like a source node. Its effects on the node density have been explored more thoroughly under the homogeneous model \cite{KEELER2}.

To calculate the Poisson mean measure over the region $\D_i$, the above node density function is integrated over the domain \cite{SKM:1995} leading to
\begin{equation}
\vec{\Lambda}_{\D_i}(u_{i+1})=\int_{u_{i}-r}^{u_{i+1}}\int_{-\psi_{u_i}(w_{i+1})}^{\psi_{u_i}(w_{i+1})} \lambda_{\D_i}(w_{i+1},\theta_{i+1})q (w_{i+1})w_{i+1}d\theta_{i+1} dw_{i+1},
\end{equation}
and we define
\begin{equation}
\vec{Q}_{\D_i}(u_{i+1})=\frac{\vec{\Lambda}_{\D_i}(u_{i+1})}{\lambda}.
\end{equation}
The joint probability density of $U_i$ and $\Theta_{i}$, that is
\begin{equation}\label{jointsleep5}
g_{i}(u_{i+1},\theta_{i+1})= \lambda_{\D_i}(u_{i+1},\theta_{i+1}) e^{-\lambda\vec{Q}_{\D_i}(u_{i+1})}.
\end{equation}

\section{Multihop distribution}
\subsection{Hop advancements}
We initially formulated this problem with the sink distance variable since greedy routing is naturally based on it. However, hop advancement is a more intuitive variable in describing message progress over a multihop path. We adopt similar notation used for the sink distance random variables, hence $\Fb$ and $\Gb$ denote the distributions under the two models. The complement of the sink distribution yields the hop distribution under both the independent and dependent models; the latter being
\begin{equation}\label{Gc}
\Gb_{i}(c_{i+1})=\left\{ \begin{array}{ll}
e^{-\lambda \vec{Q}_{i}(u_i-c_{i+1})} \quad& 0 < c_{i+1} \leq r\\
1 \quad& c_{i+1} > r\\
0 \quad& c_{i+1} \leq 0.
\end{array} \right.
\end{equation}
and its probability density defined on the absolutely continuous part of the support
\begin{equation}\label{gc}
\gb_{i}(c_{i+1}) =
  \lambda \vec{Q}'_i(u_i-c_{i+1})e^{-\lambda \vec{Q}_i(u_i-c_{i+1})},
\end{equation}
where a simple sum relates the hop and sink distance variables
\begin{equation}\label{usum}
u_i=\ell-\sum^{i}_{j=1}c_j.
\end{equation}

\subsection{Distribution of $Z_n$}
Let the random variable $Z_n$ represent the distance advanced by a message in $n$ hops
\begin{equation}\label{Zprimen5}
Z_n=\sum^n_{i=1}C_i.
\end{equation}
To calculate the distribution of $Z_n$ we use the joint probability density of the random variables $C_1$ to $C_{n}$ and $\Theta_1$ to $\Theta_{n}$, that is
\begin{equation}\label{jointg5}
\gb_{(n-1)}(c_1,.,c_{n},\theta_1,.,\theta_{n})=\prod_{i=1}^n \lambda_{\D_{i-1}}(u_{i-1}-c_{i},\theta_{i})e^{-\lambda \vec{Q}_{i-1}(u_{i-1}-c_{i})},
\end{equation}
which is defined for $c_i\in(0,r]$. The derivation of the joint probability density is similar to that under the homogeneous mode \cite{KEELER2}, which we have adapted and included in the appendix for completeness.

It follows that the distribution of message advancement after $n$ hops is expressed by
\begin{align}\label{ProbZn5}
\mathbb{P}_D(Z_{n}\leq z)=&\int_{0^+}^{\min(z,r)}dc_1\int_{-\psi_{0}(c_{1})}^{\psi_{0}(c_{1})}d\theta_1
\int_{0^+}^{\min(z-c_1,r)}dc_2\int_{-\psi_{1}(c_{2})}^{\psi_{1}(c_{2})}\dots\\
&\int_{0^+}^{\min(z-\sum_{i=1}^{n-1}c_i,r)}dc_n\int_{-\psi_{n-1}(c_{n})}^{\psi_{n-1}(c_{n})}
\gb_{(n-1)}(c_1,.,c_{n},\theta_1,.,\theta_{n}) d\theta_n\\
&+\Prob_D(C_1=0)+ \Prob_D((Z_1\leq z)\cap (C_2=0))+\dots\\ &+\Prob_D((Z_{n-1}\leq z)\cap (C_{n}=0)).
\end{align}
where
\begin{equation}\label{shortpsi5}
\psi_i(c_{i+1})=\psi_{u_i}(u_i-c_{i+1}),
\end{equation}
denotes the maximum angle for a sink distance given by the sink angle function (\ref{psi5}). The distribution of $Z_n$ under a sleep scheme is obtained by substituting the product of the joint probability densities (\ref{jointsleep5}).

The integral explicitly shown in the expression of $\mathbb{P}_D(Z_{n}\leq z)$ is the distribution of $Z_n$ conditioned on the event that all hops $C_i$ advance some positive distance. We will refer to this integral simply as the \emph{conditional distribution} of $Z_n$, and denote it by
\begin{equation}\label{ProbZnPos}
\mathbb{P}(Z_{n}\leq z|+)=\mathbb{P}(Z_{n}\leq z|C_1>0,\dots,C_n>0).
 \end{equation}

Under the independent model the joint probability density is not a function of any of the variables $\theta_1$ to $\theta_n$.  Hence, the equivalent integral can be analytically integrated over the sink angle domains \cite{KEELER1}, thus giving a simplified expression in the form of hop probability densities
\begin{align}\label{ProbZnInd5}
\mathbb{P}_I(Z_{n}\leq z)=&\int_{0^+}^{\min(z,r)}dc_1
\int_{0^+}^{\min(z-c_1,r)}dc_2\dots\\
&\int_{0^+}^{\min(z-\sum_{i=1}^{n-1}c_i,r)}\bar{f}_{0}(c_1)\dots \fb_{{n-1}}(c_n)dc_n\\
&+\Prob_I(C_1=0)+ \Prob_I((Z_1\leq z)\cap (C_2=0))+\dots\\ &+\Prob_I((Z_{n-1}\leq z)\cap (C_{n}=0)).
\end{align}

Under the dependent model, the probability of a message reaching a routing void after advancing $i$ hops
\begin{equation}
\mathbb{P}_D(C_{i+1}=0|\vec{X}_i=\vec{x}_i)=e^{-\lambda \vec{Q}_{i}(u_i)},
\end{equation}
follows. The routing void probability leads to the distribution of $Z_n$ conditioned on the event that the message meets a routing void on the last hop
\begin{align}\label{ProbZnNull5}
\nonumber\mathbb{P}_D((Z_{n}\leq z) \cap (C_{n+1}=0))&=\int_{0^+}^{\min(z,r)}dc_1\int_{-\psi_{0}(c_{1})}^{\psi_{0}(c_{1})}d\theta_1
\cdots\\
&\nonumber\int_{0^+}^{\min(z-\sum_{i=1}^{n-1}c_i,r)}dc_n\\
&\int_{-\psi_{n-1}(c_{n})}^{\psi_{n-1}(c_{n})}\gb_{(n-1)}(c_1,.,c_{n},\theta_1,.,\theta_{n})\mathbb{P}_D(C_{n+1}=0) d\theta_n.
\end{align}

Consider the event when a message does not advance, hence $X_i=X_{i+1}$. If there is a sleep scheme, a forwarding node can benefit by making multiple relay attempts. The number of different possible events soon results in the integral expressions needed to describe such a model  growing to be intractable. We restrict the integrals by assuming that a message executes only one relay attempt. 

On this note, we point out that under the homogeneous model \cite{KEELER2}, stochastic `rules of thumb' have been proposed for how many re-attempts forwarding nodes should make before considering other options (such as message backtracking). Under the inhomogeneous model, the equivalent results can be obtained by appropriately replacing the area terms with the corresponding integrals from the rescaled mean measures.

\subsection{Number of hops}
Let the random variable $N$ represent the total number of hops required for a message to reach the sink. The distribution of $N$ gives a different perspective of the performance of a routing method in a sensor network. The random variables $N$ and $Z_n$ are connected by a simple result \cite{KEELER1}, which in short says that for all $n \geq1$, the relation $\Prob_D(N\leq n)=1-\Prob_D(Z_{n-1}<\ell-r)$ holds. This results has been used to calculate the distribution of $N$ from the distribution of $Z_n$ in the homogeneous setting \cite{KEELER2}. We use this relation to calculate the equivalent results under our spatially dependent node density model (see results in Fig. \ref{Pn2a}, \ref{Pn2b} and \ref{Pn3a}). Moreover, we note that the limiting value of $\Prob_D(N\leq n)$ as $n \to \infty$ is the probability that the message ever reaches the sink, which is an important performance measure of the system, and can be derived only when all the dependence is incorporated into the model.


\section{Integration methods}
To calculate the distributions of $Z_n$ under the dependent model, a $2n$-fold integral (\ref{ProbZn5}) needs to be evaluated. For low $n$, traditional numerical integration schemes can be used. Unfortunately, integration by these methods is too slow at higher hop numbers, which motivates us to employ quasi-Monte Carlo methods.

\subsection{Quasi-Monte Carlo}
The integration description that follows is similar to the more detailed account \cite[Chap. 4]{KEELERTHESIS} where quasi-Monte Carlo methods were used to evaluate similar integrals under the homogeneous model. In recent years, these integration methods have gained much interest owing to their speed and accuracy in evaluating high dimensional integrals.

Quasi-Monte Carlo methods are based on purely deterministic sequences of quasi-random numbers such as Halton \cite{HALTON:1960} and Sobol \cite{SOBOL:1967} sequences. Mathematically, quasi-random sequences have low-discrepancy \cite{NIEDERREITER:1992,SM:1994}. Informally, such sequences appear `less random' than sequences produced by regular pseudo-random number generators as they occur more evenly spaced apart.

Previous numerical work \cite{KEELER2} has led us to use leaped Halton sequences to calculate integrals. Often integrals were also calculated via regular Monte Carlo methods to check the quasi-random approach. We also give some results based on so-called \emph{lattice rules}, which lead to specific cases of quasi-random sequences. These rules can produce well-behaving quasi-random sequences and have been a research focus in recent years owing to their ability to counter the curse dimensionality \cite{KS:2005}. The points arising from lattice rules are used in a similar manner to the quasi-Monte Carlo approach. There are many suggestions for lattice rules, but we lightly examine only one known as \emph{rank-1 lattice rule}, which over the unit hyper-cube gives the integral estimate
\begin{equation}\label{LR}
\hat{I}_nf=\frac{1}{n}\sum_{k=1}^n f(\left\{k\frac{\mathbf{z}}{n}\right\}),
\end{equation}
where the \emph{generating vector} $\mathbf{z}\in \mathbb{Z}^s$, $n$ is the number of function samples, and the braces give the fractional part in $[0,1)$.

Given a generating vector, quasi-random sequences based on such a lattice rule can be clearly produced in an exceedingly fast manner. The way to quickly evaluate an integral is by choosing a suitable generating vector. However, the drawback is that lattice rules are based on input parameters known as weights, which depend on the nature of the function. In particular, the weights depend on how the function varies with respect to all its variables. Furthermore, the choice of some lattice rules require the total number of function samples before the integral calculation starts. This differs from regular quasi-Monte Carlo methods in which the integral estimate can be calculated continually until a sufficient number of function samples has been taken.

Under lattice rules, the number of function samples also influences the choice of the generating vector. Consequently, using the most suitable generating vector may not be a simple task as it involves analyzing the function of interest. However, a thorough analysis of which weights and quasi-random sequences are the most suitable in this setting is beyond the scope of this work. We simply give some complementary results based on the rank-1 lattice rule, and leave the analysis as a future task.

The research field of lattice rules has a relatively short history, and new lattice rules are being developed continually. For more information, we refer the reader to the introductory piece by Kuo and Soan \cite{KS:2005}, and an example of lattice rules used in a financial setting \cite{GKSW:2008}. Suitable lattice rules may offer a substantially faster way of evaluating the high dimensional integrals presented here and in previous work \cite{KEELER2}. However, we focus on reducing error regardless of the chosen sequences by employing importance sampling, which has been done under the homogeneous model \cite{KEELER2}.

\subsection{Importance sampling}
To reduce the variance of the integral estimate we employ importance sampling; that is, suitably generate $C_i$ values to sample the function in key regions. Based on previous work \cite{KEELER2}, we derive an importance sampling function which is similar in form to that of the homogeneous case. We recall the expansion of the $Q$ function (\ref{approxQ}) in which we use only the first term, hence
\begin{equation}
\begin{array}{c}
  Q_{\gamma}(u)\sim q_0(u-\gamma+r)^{3/2}+O(u-\gamma+r)^{5/2}
\\
\quad\textrm{as }\quad u \rightarrow \gamma-r,
\end{array}
\end{equation}
where
\[
q_0 = \dfrac{4}{3}\left[\dfrac{2r}{\gamma (\gamma-r)}\right]^{1/2}.
\]
A change of variable $c=\gamma-u$ leads to the function
\[
\tilde{Q}_{\gamma}(c)=q_0(r-c)^{3/2},
\]
which leads to an approximate solution for the hop distribution
\[
\tilde{F}_{\gamma}(c)= e^{-\lambda \tilde{Q}_{\gamma}(c)}.
\]
The quantity $\tilde{F}_{\gamma}(0)$ is subtracted from the above expression and the result is divided by $\Delta\tilde{F}_{\gamma}=\tilde{F}_{\gamma}(c_{\max})-\tilde{F}_{\gamma}(0)$, to obtain an importance sampling function
\begin{equation}\label{ImpSampFc5}
\widehat{F}_{\gamma}(c)= \frac{1}{\Delta\tilde{F}_{\gamma}}\left[e^{-\lambda \tilde{Q}_{\gamma}(c)}-\tilde{F}_{\gamma}(0)\right],
\qquad 0\leq c \leq c_{\max}.
\end{equation}
where $c_{\max}$ is the largest hop value. The subtracting of the routing void term  has negligible effect for sufficiently large $\lambda$. The corresponding derivative needs to integrate to one, hence the rescaling step. The derivative
\begin{equation}\label{ImpSampFc5b}
\widehat{f}_{\gamma}(c)=\frac{3\lambda }{2\Delta\tilde{F}_{\gamma}}(r-c)^{1/2} e^{-\lambda \tilde{Q}_{\gamma}(c)},
\end{equation}
exists on the same domain as the sampling function. The inverse of the sampling function is given by
\begin{equation}\label{ImpSampFc5c}
\widehat{F}_{\gamma}^{-1}(t)=r-\left( \frac{-1}{\lambda a_1} \ln\left[ t\Delta\tilde{F}_{\gamma}+\tilde{F}_{\gamma}(0)\right] \right)^{3/2},
\qquad 0\leq t \leq r.
\end{equation}
Substituting a random variable from a uniform distribution, say $T\sim U(0,r)$, into the inverse of the importance sampling function
gives a random variable $C$ adhering to the importance sampling distribution (\ref{ImpSampFc5}).

After each hop is generated, the functions (\ref{ImpSampFc5b}) and (\ref{ImpSampFc5c}) should be updated by setting $\gamma$ to the current sink distance for each sample. This step was not necessary under the homogeneous model. However, under this inhomogeneous model each hop distribution varies more with respect to the sink distance. Furthermore, for high $\lambda$, the importance sampling step should improve as it is based on the independent model, and at high node density the intersection regions grow stochastically smaller \cite{KEELER1,KEELER2}.

\section{Simulation}
We compare results from routing simulations to those from our stochastic model and calculations to demonstrate that the one-hop approximation (\ref{DepSetApprox}) sufficiently captures the dependence. Given a source node sink distance $\ell$, message relaying is simulated in a circular sensor field $\mathcal{C}_{\ell}\subset \mathbb{R}^2$ of radius $\ell$ with the sink located at the origin. The fact that messages only advance towards the sink under greedy routings implies that sensor field edges do not influence the message routing. The total number of nodes per simulation is a Poisson random variable with the parameter
\begin{align*}
\Lambda(\mathcal{C}_{\ell})=\,&\lambda\int_0^{\ell}\int_{0}^{2\pi}q(u)udud\theta\\ =\,& 2\lambda\pi\ell.
\end{align*}
In simulation, a node is assigned two independent random variables $\Theta_{\textrm{S}}$ and $R_{\textrm{S}}$ corresponding to their polar coordinates in relation to the sink. To simulate node deployment such that the nodes adhere to the spatially dependent node density (\ref{inhomlambda}), both random variables are uniformly distributed
\begin{align*}
\Prob(\Theta_{\textrm{S}}\leq \theta)=\,&\dfrac{\theta}{\pi}, \qquad \theta \in[0,\pi],\\
\Prob(R_{\textrm{S}}\leq r)=\,&\dfrac{r}{\ell}, \qquad r\in[0,\ell].
\end{align*}

\section{Numerical results}
All the numerical integration and routing simulations were performed in Matlab on a standard machine. The built-in Halton sequence generator was employed for the quasi-Monte Carlo integration. For a given value of $z$, it took between $10^2$ to $10^4$ points to obtain a quasi-Monte Carlo estimate of the conditional distribution $\Prob(Z_n\leq z|+)$ where $n$ ranged from $2$ to $20$. The actual number of points depends on the number of hops $n$; more hops require more points. The exact relationship between the required number of points and $n$ is not known as the number of points also depends on $\lambda$, but future analytic empirical work may shed light on the relationship. The importance sampling step improved the rate of convergence, particularly for high $\lambda$. All calculations took no longer than an hour to complete, and usually considerably less.

We compared the dependent model to routing simulations of various ensemble sizes with and without a blinking sleep scheme. Generally, between $10^3$ to $10^5$ routing simulations were required. Similar to the integration process, higher hop numbers required more simulations to give converged results. The routing simulations are based on the same assumptions made in the mathematical model.

We observed under the inhomogeneous model that more function samples and routing simulations are needed to give similarly converged results compared to those obtained under the homogeneous model. An extensive empirical investigation is needed to see how fast quasi-Monte Carlo methods are compared to routing simulations. Also, more empirical and theoretical evidence is needed to elucidate the advantages and disadvantages of calculating probabilistic behaviour of greedy routing via our model.

\begin{figure}
\begin{center}
\includegraphics[scale=0.6]{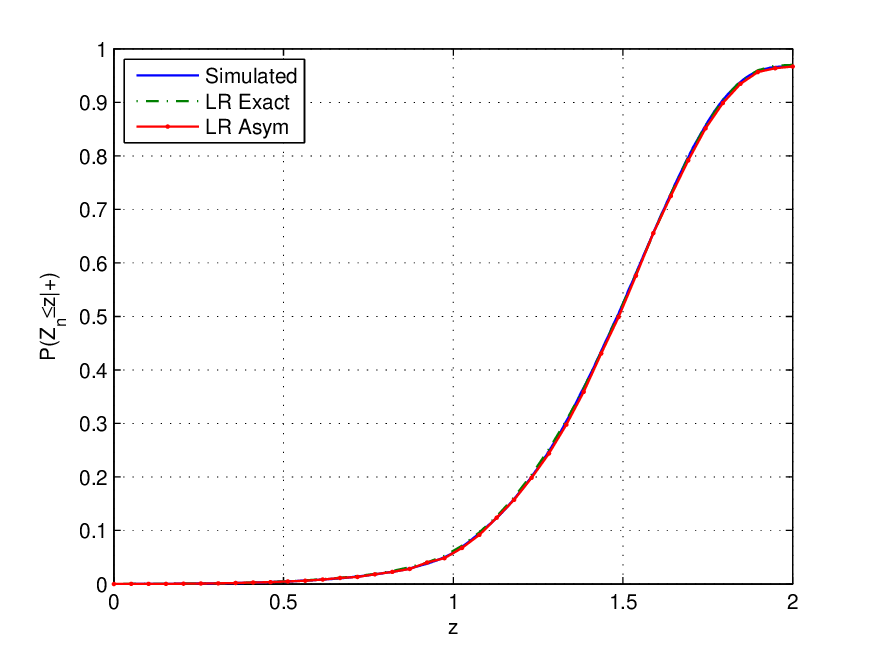}
\caption{Results of $\Prob(Z_2\leq z|+)$ via lattice rules integration (based on analytic and asymptotic expressions) and simulations ($p=1$,$\lambda=3\ell$, $r=1$ and  $\ell=10$).\label{Comp2HopLRExactAsymVsEmpZ}}
\end{center}
\end{figure}


The lattice rule approach was only used to calculate one set of results (see Fig. \ref{Comp2HopLRExactAsymVsEmpZ}). We used generator vectors based on fixed lattice rules for $2^{10}$ function sample points and equal weights \cite{KuoWeb}. We applied ten random shifts to the lattice rules; see Giles et al. \cite{GKSW:2008} for an example. This approach generally performed well, however, under importance sampling sometimes erratic results arose. This is possibly due to an over bias in function sampling or an unknown numerical artefact. A more thorough examination is needed, but we believe that the preliminary results using lattice rules are promising.

We found that the results based on elliptic integral functions could be replaced with results that used the three-term approximations (\ref{approxpsi5}) and (\ref{approxQ}) with no discernible loss of accuracy (see the plots in Fig. \ref{Comp2HopLRExactAsymVsEmpZ}). Evaluating the approximations is faster as the expressions only involve elementary functions. Consequently, the remaining results are based on these approximations (Fig. \ref{Pn2a} to Fig. \ref{Pn3a}).

\begin{figure}
\begin{center}
\includegraphics[scale=0.6]{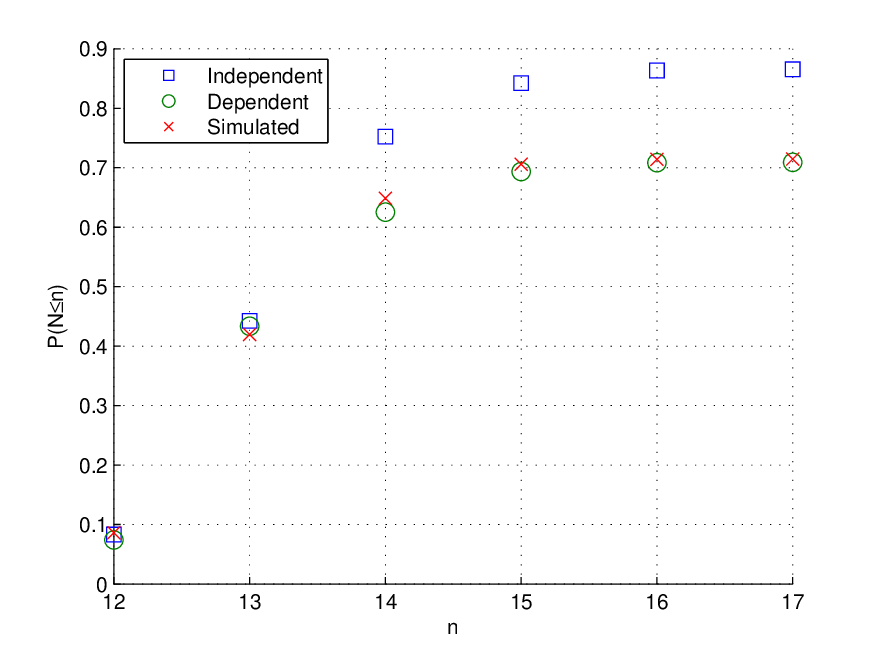}
\caption{Independent and dependent model results of $\Prob(N\leq n)$ compared to simulations ($p=1$, $\lambda=2\ell$, $r=1$ and  $\ell=10$).\label{Pn2a}}
\end{center}
\end{figure}

\begin{figure}
\begin{center}
\includegraphics[scale=0.6]{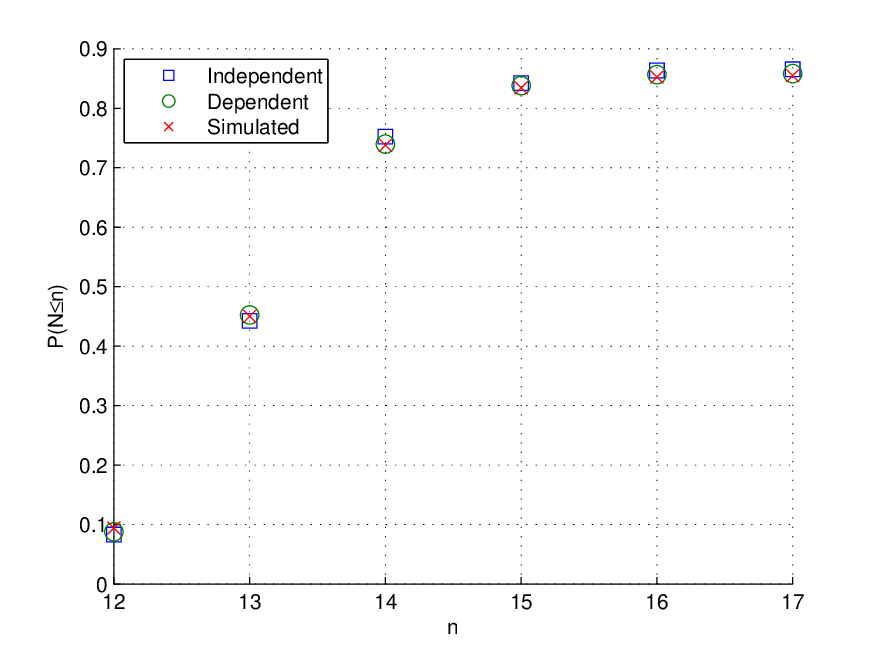}
\caption{Independent and dependent model results of $\Prob(N\leq n)$ compared to simulations ($p=0.1$, $\lambda=2\ell$, $r=1$ and  $\ell=10$).\label{Pn2b}}
\end{center}
\end{figure}

\begin{figure}
\begin{center}
\includegraphics[scale=0.6]{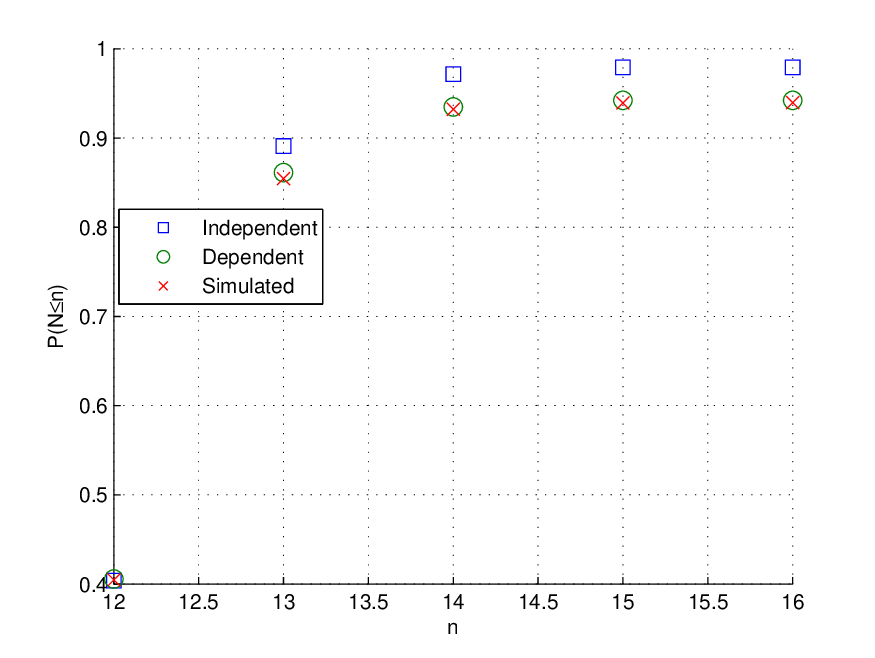}
\caption{Independent and dependent model results of $\Prob(N\leq n)$ compared to simulations ($p=1$, $\lambda=3\ell$, $r=1$ and  $\ell=10$).\label{Pn3a}}
\end{center}
\end{figure}

We calculated the distribution of $N$, and included a simple sleep scheme to see its influence on the path dependence. We compared the independent model to the dependent model by varying $p$ (and accordingly $\alpha$, hence holding $\lambda$ constant). The difference between the two models has an accumulative effect on $N$, hence it serves as a good indicator of path dependence.

Under the independent model, the distribution $\Prob(N\leq n)$ gives greater values for each $n$ than the equivalent result under the dependent model. Also, the path dependence clearly lessens as $p$ approaches zero (compare Fig. \ref{Pn2a} and Fig. \ref{Pn2b}). These results are analogous to those under the homogeneous density model \cite{KEELER2}.

For large $\lambda$, the difference between the two models is less. As is the case for the homogeneous, a larger density results in the next forwarding node being closer to the sink, thus reducing the intersection of the feasible regions and lessening the path dependence. Moreover, under the inhomogeneous model hops grows stochastically larger due to the increasing node density. Thus, we believe that the path dependence continues to decrease stochastically as the message approaches the sink.

In conclusion, the results reveal that the dependent model clearly captures the path dependence. The resulting expressions, both involving the elliptic integrals and the asymptotic expansions, give results that closely agree with simulations. We believe more numerical investigation is needed to choose the most appropriate quasi-random sequences in this setting.


\section{Future work}
In realistic settings, the constant node density assumption may often not be appropriate. Under the spatially dependent model, the density function was chosen such that it was simple enough for obtaining analytic and asymptotic means, while still being a plausible node placement scenario. Other suggestions exist such as the node density decaying exponentially or according to some inverse power of the sink distance.

Furthermore, under our model the sink was located at the maximum of the node density. Placing the sink at an arbitrary point in the sensor field results in the node density being dependent on the sink angle. This is an additional increase in the complexity of the density function. Moreover, the angle of an individual node would not be a uniformly distributed random variable, thus importance sampling might be needed when integrating over the angle domains. Consequently, these suggestions may result in analytic and asymptotic mean measures that can be used to model more realistic node deployment models.

Further investigation of the asymptotic approximations are needed. The approximation may break down when the radius is large compared to the sink distance. For a constant radius model, the lengths can always be rescaled with respect to the transmission radius. However, this may not be possible for a randomly varying radius model. We stress that including random transmission radii into our model would be an interesting and realistic model extension in itself.

We presented some integration results based on lattice rules. Despite these methods giving mostly agreeable results, further investigation is needed to gauge which lattice rules and quasi-random sequences are the most suitable for the integrals that arise from our model. This investigation would be both analytic and numerical in nature. This work may lead to our mathematical model considerably outperforming regular routing simulations.

Finally, an attractive feature of regular Monte Carlo methods is that the error is obtained by estimating the variance of the integral. Conversely, quasi-Monte Carlo methods lack a practical way of estimating the error, despite them generally have a faster convergence rate. The idea of `randomized' quasi-Monte Carlo methods seeks to combine the advantages of both approaches. Consequently, a future task lies in investigating these hybrid methods in evaluating hop integrals.

\section{Conclusion}
We presented a tractable inhomogeneous spatially dependent density model. Inspired by previous work, we developed and examined a greedy routing model that incorporates both sink and path dependence. Moreover, the spatially dependent density model verified that the formulation of the homogeneous model can be extended to an inhomogeneous case. This model is an alternative means of ascertaining the stochastic characteristics of greedy routing in sensor and ad hoc networks.

We used asymptotic methods to derive accurate approximations for hop length moments and the mean measure for our spatially dependent node density model. We used quasi-Monte Carlo methods and recently developed lattice rules coupled with importance sampling to estimate the resulting high dimensional integrals. For a sufficient number of function samples, all the results agreed admirably with those obtained by routing simulations.

Finally, we included a sleep scheme to demonstrate its effects on the local node density and the path dependence. For systems with a low $p$, the results imply that the independent model can be used, thus reducing computation time in calculating stochastic properties of the system.

\section{Acknowledgements}
The author thanks Peter G. Taylor at the University of Melbourne for valuable discussions and advice. A further thanks to Frances Kuo and Ian Sloan at the University of New South Wales for kindly explaining the subtleties of lattice rules. The author would also like to acknowledge the support of the Australian Research Council Centre of Excellence for Mathematics and Statistics of Complex Systems and the David Lachlan Hay Memorial Fund.

\bibliographystyle{siam}
\bibliography{References}

\begin{thebibliography}{10}

\bibitem{DLMF:2010}
{\em {Digital Library of Mathematical Functions}}, {Accessed on the 15th of
  January} 2010.
\newblock http://dlmf.nist.gov/.

\bibitem{AKYILDIZ:2002}
{\sc I.~F. Akyildiz, S.~Weilian, Y.~Sankarasubramaniam, and E.~Cayirci}, {\em A
  survey on sensor networks}, IEEE Communications Magazine, 40 (2002),
  pp.~102--114.

\bibitem{BB1:2009}
{\sc F.~Baccelli and B.~Blaszczyszyn}, {\em Stochastic Geometry and Wireless
  Networks, Volume I - Theory}, vol.~1, NOW Publishers, Delft, The Netherlands,
  2009.

\bibitem{BB2:2009}
\leavevmode\vrule height 2pt depth -1.6pt width 23pt, {\em Stochastic Geometry
  and Wireless Networks, Volume II - Theory}, vol.~1, NOW Publishers, Delft,
  The Netherlands, 2009.

\bibitem{STOJ:1999}
{\sc P.~Bose, P.~Morin, I.~Stojmenovic, and J.~Urrutia}, {\em Routing with
  guaranteed delivery in ad hoc wireless networks}, in 3rd International
  Workshop on Discrete Algorithms and Methods for Mobile Computing and
  Communications, ACM, 1999, pp.~48--55.

\bibitem{CARLSON:1979}
{\sc B.~Carlson}, {\em Computing elliptic integrals by duplication}, Numerische
  Mathematik, 33 (1979), pp.~1--16.

\bibitem{CARLSON:1995}
\leavevmode\vrule height 2pt depth -1.6pt width 23pt, {\em Numerical
  computation of real or complex elliptic integrals}, Numerical Algorithms, 10
  (1995), pp.~13--26.

\bibitem{CHONG:2003}
{\sc C.~Chong and S.~P. Kumar}, {\em Sensor networks: Evolution, opportunities
  and challenges}, Proceedings of the IEEE, 91 (2003), pp.~1274--1256.

\bibitem{GKSW:2008}
{\sc M.~B. Giles, F.~Y. Kuo, I.~H. Sloan, and B.~J. Waterhouse}, {\em
  Quasi-monte carlo for finance applications}, in Proceedings of the 14th
  Biennial Computational Techniques and Applications Conference, CTAC-2008,
  G.~N. Mercer and A.~J. Roberts, eds., vol.~50 of ANZIAM Journal, Nov 2008,
  pp.~C308--C323.

\bibitem{HABDF:2009}
{\sc M.~Haenggi, J.~Andrews, F.~Baccelli, O.~Dousse, and M.~Franceschetti},
  {\em Stochastic geometry and random graphs for the analysis and design of
  wireless networks}, IEEE Journal on Selected Areas in Communications, 27
  (2009), pp.~1029--1046.

\bibitem{HALL:1988}
{\sc P.~Hall}, {\em Intoduction to the Theory of Coverage Process}, John Wiley
  and Sons., 1st~ed., 1988.

\bibitem{HALTON:1960}
{\sc J.~Halton}, {\em On the efficiency of certain quasi-random sequences of
  points in evaluating multi-dimensional integrals}, Numerische Mathematik, 2
  (1960), pp.~84--90.

\bibitem{IA:2004}
{\sc M.~Ishizuka and M.~Aida}, {\em Performance study of node placement in
  sensor networks}, in Proceedings 24th International Conference on Distributed
  Computing Systems Workshops, 2004., March 2004, pp.~598--603.

\bibitem{IA:2007}
\leavevmode\vrule height 2pt depth -1.6pt width 23pt, {\em Stochastic node
  placement improving fault tolerance in wireless sensor networks}, Electronics
  and Communications in {Japan}, 90 (2007), pp.~2181--2191.

\bibitem{KARP:2000}
{\sc B.~Karp and H.~T. Kung}, {\em Greedy perimeter stateless routing for
  wireless networks}, in Sixth Annual ACM/IEEE International Conference on
  Mobile Computing and Networking (MobiCom 2000), 2000, pp.~243--254.

\bibitem{KEELERTHESIS}
{\sc H.~P. Keeler}, {\em Stochastic Routing Models in Sensor Networks}, PhD
  thesis, University of Melbourne, 2010.

\bibitem{KEELER1}
{\sc H.~P. Keeler and P.~G. Taylor}, {\em A stochastic analysis of a greedy
  routing scheme in sensor networks}, SIAM Journal on Applied Mathematics, 70
  (2010), pp.~2214--2238.

\bibitem{KEELER2}
{\sc H.~P. Keeler and P.~G. Taylor}, {\em A model framework for greedy routing
  in a sensor network with a stochastic power scheme}, ACM Transactions on
  Sensor Networks, 7 (2011).

\bibitem{KUHN:2003}
{\sc F.~Kuhn, R.~Wattenhofer, Y.~Zhang, and A.~Zollinger}, {\em Geometric
  ad-hoc routing: of theory and practice}, in 22 nd Annual Symposium on
  Principles of Distributed Computing, 2003.

\bibitem{KULLBACK:1959}
{\sc S.~Kullback}, {\em Information Theory and Statistics}, Wiley, 1st~ed.,
  1959.

\bibitem{KuoWeb}
{\sc F.~Kuo}, {\em
  {http://web.maths.unsw.edu.au/$\sim$fkuo/lattice/index.html}}.
\newblock Personal Website, {Accessed on the 4th of August} 2009.

\bibitem{KS:2005}
{\sc F.~Y. Kuo and I.~H. Sloan}, {\em Lifting the curse of dimensionality},
  Notices of the {AMS}, 52 (2005), pp.~1320--1328.

\bibitem{MAUVE:2001}
{\sc M.~Mauve, A.~Widmer, and H.~Hartenstein}, {\em A survey on position-based
  routing in mobile ad hoc networks}, IEEE Network, 15 (2001), pp.~30--39.

\bibitem{NIEDERREITER:1992}
{\sc H.~Niederreiter}, {\em \'{Random Number Generation and Quasi-Monte Carlo
  Methods}}, SIAM, 1992.

\bibitem{MRM:2009}
{\sc M.~Pallavi, S.~S. Ram, and D.~Manjunath}, {\em Path coverage by a sensor
  field: The nonhomogeneous case}, ACM Transactions on Sensor Networks, 5
  (2009), pp.~1--26.

\bibitem{SOBOL:1967}
{\sc I.~Sobol}, {\em The distribution of points in a cube and the approximate
  evaluation of integrals}, U.S.S.R. Computational Mathematics and Mathematical
  Physics, 7 (1967), pp.~86--112.
\newblock (In Russian).

\bibitem{SM:1994}
{\sc J.~Spanier and E.~H. Maize}, {\em Quasi-random methods for estimating
  integrals using relatively small samples}, SIAM Review, 36 (1994),
  pp.~18--44.

\bibitem{STOJ:2002}
{\sc I.~Stojmenovic}, {\em Position-based routing in ad hoc networks}, IEEE
  Communications Magazine, 40 (2002), pp.~128--134.

\bibitem{SKM:1995}
{\sc D.~Stoyan, W.~Kendall, and J.~Mecke}, {\em Stochastic Geometry and its
  Applications}, Wiley, 2nd~ed., 1995.

\bibitem{TUBAISHAT:2003}
{\sc M.~Tubaishat and S.~Madria}, {\em Sensor networks: an overview}, IEEE
  Potentials, 22 (2003), pp.~20--23.

\bibitem{WONG:1989}
{\sc R.~Wong}, {\em Asymptotic Approximations to Integrals}, Academic Press,
  New York, 1989.

\bibitem{ZORZI1:2003}
{\sc M.~Zorzi and R.~R. Rao}, {\em Geographic random forwarding ({GeRaF}) for
  ad hoc and sensor networks: multihop performance}, IEEE Transactions on
  Mobile Computing, 2 (2003), pp.~337--348.
\newblock 1536-1233.

\end{thebibliography}

\appendix

\section{Derivation of $\vec{Q}_i(u_{i+1})$}
We outline how to calculate the rescaled mean measure for the region feasible region under the dependent model. The method is akin to calculating the equivalent area function $\vec{A}_i(u_{i+1})$ in the homogeneous case \cite{KEELER1}. In fact, the derivation of $\vec{Q}_i(u_{i+1})$ is included for completeness, and we refer the reader to previous work \cite{KEELER1} for further details.

We use the function $\Delta\psi(u_2)$ again to describe the angular width of the intersection of the source and current feasible regions. Subsequently, the rescaled mean measure on this intersection region is given by
\[
Q_{1\setminus0}(u_2)= \int^{u_2}_{\ell-r}\Delta\psi(w_2)dw_2,
\]
which leads to the rescaled mean under the dependent model
\[
\vec{Q}_1(u_{2})=Q_{u_1}(u_{2})-Q_{1\setminus 0}(u_{2}).
\]
\begin{figure}[h]
\begin{center}
\begin{overpic}[scale=0.52]{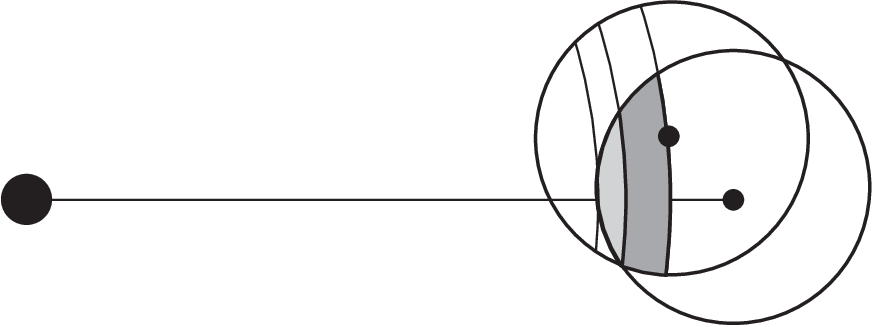}
\put(80,17){$X_{0}$}
\put(77,24){$X_{1}$}
\put(64,3){$X_{01}$}
\end{overpic}
\caption{The form of intersection region depends on the $u_2$
interval and the location of $X_{01}$.\label{u2region}}
\end{center}
\end{figure}

It can be shown that the intersection angle expression is
\[
\Delta\psi(u_2)=2\psi_{\ell}(u_2)\mathbf{I}^-_{X_{01}}, \qquad \ell-r\leq
u_2\leq u_{01},
\]
where $\mathbf{I}^-_{{01}}$ is an indicator function for when the intersection-point $X_{01}$ is below the baseline that runs from $X_0$ to $X_S$ (refer to Fig. \ref{u2region}), and $u_{01}$ is the sink distance of $X_{01}$. On the second interval we obtain the
intersection angle expression
\[
\nonumber \Delta\psi(u_2)=\psi_{\ell}(u_2)+\psi_{u_1}(u_2)-\theta_1,\qquad
u_{01}\leq u_2\leq u_1.
\]

Recall under the independent model the rescaled mean integral
\[
Q_{\gamma}(u)=
2\int_{\gamma-r}^u \psi_{\gamma}(w)d w,
\]
where the angle function
\[
\psi_{\gamma}(u)= \arccos\left(\frac{u^2+\gamma^2-r^2}{2u\gamma}\right).
\]
Thus, on the first interval, $[\ell-r,u_{01}]$, we have the rescaled mean
\begin{align*}
  Q_{1\setminus0}(u_2)=\,&2\int^{u_2}_{\ell-r}\psi_{\ell}(w_2)dw_2\mathbf{I}^-_{X_{01}}\\
  =\,& Q_{\ell}(u_2)\mathbf{I}^-_{X_{01}}.
\end{align*}
On the second interval, $[u_{01},u_1]$, we have the slightly more complicated rescaled mean expression
\begin{align*}
Q_{1\setminus0}(u_2)=\,&\int^{u_2}_{\ell-r}\left[\psi_{\ell}(w_2)+\psi_{u_1}(w_2)-\theta_1\right]dw_2+Q_{\ell}(u_{01})\mathbf{I}^-_{X_{01}},\\
=\,& \frac{1}{2}\left(Q_{\ell}(u_2)+Q_{u_1}(u_2)+2\theta_1\left[u_{01}-u_2\right]\right)\\
&+\frac{1}{2}\left(Q_{\ell}(u_{01})[2\mathbf{I}^-_{X_{01}}-1]-Q_{u_1}(u_{01})\right).
\end{align*}
This approach naturally extends to the rescaled mean on the intersection of any two feasible regions. Consequently, for $i\geq 1$, under the one-hop dependent model the rescaled mean measure on the feasible region is given by
\[
\vec{Q}_i(u_{i+1})=Q_{u_i}(u_{i+1})-Q_{i\setminus {i-1}}(u_{i+1}).
\]


\section{Joint probability density}
We assume there is no sleep scheme, and note that to include a sleep scheme entails substituting the corresponding node density function (\ref{sleeplambda}) into the joint probability density.

We consider the probability density of $\Theta_{i}$ conditioned on the event $U_i=u_i$. Under our inhomogeneous Poisson model, the angle of any node is distributed uniformly around the sink (in the regions where nodes can exist). Hence, the conditional probability density under the dependent model
\[
  g_{u_1}(\theta_1|U_1=u_1)=
  \dfrac{\mathbf{I}_{\D_0}(u_{1},\theta_{1})}{2\psi_{u_0}(u_1)},
\]
follows, and this expression also applies to the independent model as it has the same feasible region. We introduce the function $\Psi_{\vec{x}_i}(u_{i+1})$ to denote the total angular width of the feasible region given $U_{i+1}=u_{i+1}$ and the path $\vec{X}_i=\vec{x}_i$. The conditional probability density
\[
  g_{u_{i+1}}(\theta_{i+1}|U_{i+1}=u_{i+1}) =
  \dfrac{\mathbf{I}_{\D_i}(u_{i+1},\theta_{i+1})}{\Psi_{\vec{x}_i}(u_{i+1})},
\]
follows. Under the independent model the angular width function simplifies to \[\Psi_{\vec{x}_i}(u_{i+1})=2\psi_{u_i}(u_{i+1}).\]
The rescaled mean of the feasible region written as an integral
\[
\vec{Q}_{i}(u_{i+1})= \int^{u_{i+1}}_{u_i-r}\Psi_{\vec{x}_i}(w_{i+1})dw_{i+1},
\]
gives the derivative of the rescaled mean measure
\[
\vec{Q}'_{i}(u_{i+1})= \Psi_{\vec{x}_i}(u_{i+1}).
\]
The probability density
\[
g_{\vec{x}_i}(u_{i+1})=\lambda \Psi_{\vec{x}_i}(u_{i+1})e^{-\lambda \vec{Q}_{i}(u_{i+1})},
\]
follows. Hence, the joint probability density of the two random variables $U_{i+1}$ and $\Theta_{i+1}$ is given by
\begin{align*}
g_{i}(u_{i+1},\theta_{i+1})&=  g_{\vec{x}_{i}}(u_{i+1}) g_{u_{i+1}}(\theta_{i+1}|U_{i+1}=u_{i+1},\vec{X}_{i}=\vec{x}_{i})\\
&=\lambda_{\D_{i}}(u_{i+1},\theta_{i+1})e^{-\lambda \vec{Q}_{i}(u_{i+1})}
\end{align*}
where the spatially dependent node density function has been introduced
\[\lambda_{\D_{i}}(u_{i+1},\theta_{i+1})=\lambda\mathbf{I}_{\D_{i}}(u_{i+1},\theta_{i+1}).\]
The joint probability density of the random variables $U_1$ to $U_{n}$ and $\Theta_1$ to $\Theta_{n}$
\begin{align}
g_{(n-1)}(u_1,.,u_{n},\theta_1,.,\theta_{n})= \prod_{i=1}^n \lambda_{\D_{i-1}}(u_{i},\theta_{i})e^{-\lambda \vec{Q}_{i-1}(u_{i})},
\end{align}
follows, or in terms of hop advancements the equivalent expression
\[
\gb_{(n-1)}(c_1,.,c_{n},\theta_1,.,\theta_{n})=\prod_{i=1}^n \lambda_{\D_{i-1}}(u_{i-1}-c_{i},\theta_{i})e^{-\lambda \vec{Q}_{i-1}(u_{i-1}-c_{i})}.
\]

\end{document}